\documentclass[12pt]{article}

\usepackage{amsmath}
\usepackage{amsfonts}
\usepackage{amssymb}
\usepackage{graphicx, color}
\usepackage[colorlinks]{hyperref}
\usepackage{epsfig}
\usepackage{epstopdf} 
\usepackage{mathtools}
\usepackage{comment}
\usepackage[shortlabels]{enumitem}

\newtheorem{theorem}{Theorem}[section]
\newtheorem{lemma}[theorem]{Lemma}

\newtheorem{remark}[theorem]{Remark}
\newtheorem{example}[theorem]{Example}

\newtheorem{definition}[theorem]{Definition}

\newcommand{\recip}{\displaystyle\frac{1}}
\newcommand{\Frac}{\displaystyle\frac}
\newcommand{\nl}{\\[5pt]}

\newcommand{\inR}{\in\R}
\newcommand{\Int}{\displaystyle\int}
\newcommand{\R}{\mathbb{R}}
\newcommand{\CL}{\mathcal{L}}

\newcommand{\tu}{\tilde{u}}
\newcommand{\tv}{\tilde{v}}
\newcommand{\ta}{\tilde{a}}
\newcommand{\tH}{\tilde{H}}
\newcommand{\tF}{\tilde{F}}
\newcommand{\tf}{\tilde{f}}
\newcommand{\tJ}{\tilde{J}}
\newcommand{\tK}{\tilde{K}}

\newcommand{\tsigma}{\tilde{\sigma}}

\usepackage{mathrsfs}
\usepackage{array}

\usepackage{xcolor}

\begin{document}
\title{Quantitative Stability of Generalized $p$-Area Minimizing Surfaces}

\author{ 
    \qquad Amir Moradifam\footnote{Department of Mathematics, University of California, Riverside, California, USA. E-mail: amirm@ucr.edu. Amir Moradifam is supported by NSF grants DMS-1715850 and DMS-1953620.}
    \qquad Gerardo Orozco-Fernandez \footnote{Department of Mathematics, University of California, Riverside, California, USA. E-mail: goroz003@ucr.edu.   }
}

\date{\today}
\smallbreak \maketitle

\begin{abstract}
We study the stability of $p$-area minimizing surfaces in the Heisenberg group under perturbations of the weight function and the drift vector field in generalized least gradient problems of the form
\[
\inf_{w\in BV_0(\Omega)}
\int_\Omega \left(a(x)|Dw+F(x)|+H(x)w\right)\,dx.
\]
Owing to the lack of strict convexity, establishing stability of minimizers is challenging. We derive quantitative stability estimates for minimizers. In particular, under suitable nondegeneracy and geometric assumptions, we obtain $L^1$ and $W^{1,1}$ stability estimates with respect to perturbations of the weight function $a$ and the drift vector field $F$. We further establish unified quantitative stability estimates under simultaneous perturbations of all principal parameters, namely $a$, $F$, and $H$. Numerical simulations illustrating the stability theory are also presented.
\end{abstract}

\section{Introduction}

Least gradient problems and variational models with linear growth have been extensively studied in geometric analysis, the calculus of variations, and inverse problems. Classical least gradient problems arise in the study of functions minimizing the total variation functional and are closely connected to minimal surfaces, geometric measure theory, and free discontinuity problems \cite{G,sternberg_ziemer92,sternberg_ziemer,sternbergZiemer93,MRD}. In recent years, considerable attention has also been devoted to generalized least gradient problems involving anisotropic energies, weighted functionals, and drift terms \cite{AB,And4,Mazon2,DLPT,Gorny}. Such problems naturally appear in conductivity imaging, free boundary problems, and sub-Riemannian geometry.

In the Heisenberg group, minimizers of the $p$-area functional provide a natural analogue of minimal surfaces in Euclidean geometry and play a fundamental role in pseudohermitian and sub-Riemannian analysis \cite{CHMY,CHP,P,FSS,GN}. The associated Euler--Lagrange equation is highly degenerate and exhibits substantial analytical and geometric difficulties. Questions concerning existence, uniqueness, regularity, and geometric structure of $p$-area minimizing surfaces have been studied extensively in \cite{CH,CHY,LM,MR,MR2,PSTV}. In particular, existence and structural properties of minimizers for generalized $p$-area functionals were established in \cite{LM,MR,MR2}, while uniqueness and geometric aspects were investigated in \cite{CH,CHY}.

A closely related class of variational problems appears in conductivity imaging and hybrid inverse problems, where minimizers of weighted least gradient functionals arise naturally in the reconstruction of anisotropic conductivities from interior measurements \cite{HMN,JMN,MNT,MNTa_SIAM,NTT07,NTT08,NTT10,NTT11,ST}. In these problems, the geometry of level sets and the associated current density vector fields play a central role. Motivated by these connections, several works have developed duality and geometric approaches for least gradient type problems \cite{Mo,Mo1,JMN}.

In our earlier work \cite{MO}, we initiated a quantitative stability theory for $p$-area minimizing surfaces in the Heisenberg group with respect to perturbations of the prescribed $p$-mean curvature $H$. More precisely, for the variational problem
\begin{equation}\label{eq:intro1}
\inf_{w\in BV_0(\Omega)}
\int_\Omega \left(a(x)|Dw+F(x)|+H(x)w\right)\,dx,
\end{equation}
we established quantitative estimates describing the dependence of minimizers on perturbations of the forcing term $H$. The main idea in \cite{MO} was to exploit a Rockafellar--Fenchel duality framework in order to associate to every minimizer a canonical flux vector field
\[
J=\sigma(Du+F),
\qquad
\sigma=\frac{a}{|Du+F|},
\]
satisfying
\[
\nabla\cdot J=H
\]
in the sense of distributions. Stability estimates for the flux variable were then transferred to the minimizers themselves through a geometric analysis of level sets combined with the coarea formula. Under suitable structural assumptions on the level sets, quantitative $L^1$ and $W^{1,1}$ stability estimates were obtained.

The purpose of the present paper is to continue and extend the stability theory developed in \cite{MO}. Here we investigate the dependence of minimizers on perturbations of the remaining structural parameters of the model, namely the weight function $a$ and the drift vector field $F$, and ultimately obtain combined stability estimates with respect to simultaneous perturbations of $a$, $F$, and $H$.

More precisely, we study minimizers of \eqref{eq:intro1} in dimensions $n=2,3$, where $\Omega\subset\mathbb{R}^n$ is a bounded Lipschitz domain, the weight $a$ satisfies suitable nondegeneracy assumptions, and $F$ is a prescribed vector field. As in \cite{MO}, the analysis is carried out at the level of the associated flux vector field
\[
J=\sigma(Du+F),
\]
which allows one to separate the variational and geometric aspects of the problem.

We first establish quantitative stability estimates for the flux vector field and minimizers with respect to perturbations of the weight $a$. Under natural boundedness and admissibility assumptions, we prove estimates of the form
\[
\|J-\tilde J\|_{L^1(\Omega)}
\leq
C\|a-\tilde a\|_{L^\infty(\Omega)}^{1/2}.
\]
together with corresponding $L^1$ stability estimates for admissible minimizers. Under additional regularity assumptions and geometric hypotheses on the structure of level sets, we further derive $W^{1,1}$ stability estimates with rate $1/4$.

We then turn to perturbations of the drift vector field $F$. In this setting, the geometry of the level sets interacts directly with the perturbation of the underlying vector field, leading to a more delicate stability mechanism. Assuming that the drift fields are conservative, we establish quantitative stability estimates in terms of $\|F-\tilde F\|_{L^1(\Omega)}$ for the associated fluxes, minimizers, and the corresponding coefficients $\sigma$ and $\tilde\sigma$. The proofs combine duality arguments, weak integration by parts formulas, and geometric estimates along foliations of level sets.

Finally, combining the results obtained here with the stability theory established in \cite{MO}, we derive quantitative stability estimates under simultaneous perturbations of all principal parameters of the model. In particular, we obtain unified $L^1$ and $W^{1,1}$ stability estimates for admissible minimizers with respect to perturbations of $a$, $F$, and $H$.

One of the main features of our approach is that the stability theory is developed through the flux variable rather than directly at the level of minimizers. This viewpoint provides a unified framework which naturally connects variational structure, duality methods, and geometric properties of level sets. The analysis also highlights the role played by admissibility and geometric regularity assumptions in transferring stability from the flux variable to the minimizers themselves.

In addition to the theoretical results, we present numerical simulations based on the convergent algorithm developed in \cite{CMO}. The simulations illustrate the stability of minimizers under noisy perturbations of the data and are consistent with the analytical estimates established in the paper.

The paper is organized as follows. In Section 2 we establish stability results with respect to perturbations of the weight function $a$, including quantitative estimates for the flux vector field, minimizers, and the coefficient $\sigma$. In Section 3 we study stability with respect to perturbations of the drift vector field $F$. Section 4 combines the results of this paper with those obtained in \cite{MO} in order to derive unified stability estimates under simultaneous perturbations of $a$, $F$, and $H$. The final subsection presents numerical simulations illustrating the stability theory.

\section{Stability with respect to the weight $a$}

Let $\Omega \subset \R^n$ be a connected, bounded, open set with Lipschitz boundary, and consider the variational problem
\begin{equation}\label{mainmin0}
    \inf_{w\in H^1_0(\Omega)} \int_\Omega \left(a|Dw+F|+Hw\right).
\end{equation}
More generally, minimization over the affine space
\[
H^1_f(\Omega):=\{w\in H^1(\Omega): w=f \ \text{on}\ \partial\Omega\}
\]
can be reduced to \eqref{mainmin0} through the change of variables $w=w_0+f$, where $w_0\in H^1_0(\Omega)$ and $f\in H^1(\Omega)\cap H^1_f(\Omega)$.

As in \cite{MO}, our analysis is based on the existence of an associated vector field $J$ that determines the direction of $Du+F$ in $\Omega$ for every minimizer $u$. By Rockafellar--Fenchel duality \cite{ET}, the dual problem corresponding to \eqref{mainmin0} takes the form
\begin{equation}
   (D)\hspace{0.5cm}
   \sup \left\{\langle F,b\rangle:\ b\in \mathcal{D}_0 \ \text{and}\ |b|\leq a \ \text{a.e. in}\ \Omega \right\},
\end{equation}
where
\begin{equation}
    \mathcal{D}_0:=
    \left\{
    b\in (L^2(\Omega))^n:
    \int_{\Omega} \nabla u\cdot b + Hu =0,
    \ \text{for all}\ u\in H^1_0(\Omega)
    \right\}.
\end{equation}

Observe that $b\in (L^\infty(\Omega))^n\cap \mathcal{D}_0$ if and only if
\[
\nabla\cdot b = H
\qquad \text{a.e. in}\ \Omega.
\]
We refer the reader to Section 3.1 of \cite{MR} for further details.

\begin{theorem}\label{Structure1}
\cite{MR}
Let $\Omega$ be a bounded domain in $\R^n$, let $F,H\in L^2(\Omega)$, and let $a\in L^2(\Omega)$ be positive. Assume that \eqref{mainmin0} is bounded from below. Then the duality gap is zero and the dual problem $(D)$ admits a solution. In particular, there exists a vector field $J\in \mathcal{D}_0$ satisfying $|J|\leq a$ almost everywhere in $\Omega$ such that
\begin{equation}\label{dualityGap2}
\inf_{u\in H^1_0(\Omega)}
\int_{\Omega}
\left(a|Du+F|+Hu\right)\,dx
=
\langle F,J\rangle.
\end{equation}
Moreover, for every minimizer $u$ of \eqref{mainmin0},
\begin{equation}\label{directionParallel2}
a\,\frac{Du+F}{|Du+F|}
=
J,
\qquad
|Du+F|\text{-a.e. in }\Omega.
\end{equation}
\end{theorem}

Observe that $u$ satisfies the Euler--Lagrange equation
\begin{equation}\label{EL}
    \nabla\cdot\left(a\frac{Du+F}{|Du+F|}\right)=H,
    \qquad
    u|_{\partial\Omega}=0.
\end{equation}
Moreover, there exists a vector field $J$, independent of the particular minimizer $u$, such that
\[
a\frac{Du+F}{|Du+F|}
=
J,
\qquad
|Du+F|\text{-a.e. in }\Omega,
\]
with
\[
\nabla\cdot J = H.
\]
Furthermore, one may write
\[
J(x)=\sigma(x)(Du(x)+F(x)),
\qquad
|J|=a,
\qquad
|Du+F|\text{-a.e. in }\Omega,
\]
for some measurable function $\sigma\geq 0$. Consequently,
\[
a
=
|J|
=
\sigma|Du+F|
=
J\cdot\frac{Du+F}{|Du+F|},
\qquad
|Du+F|\text{-a.e. in }\Omega.
\]
Likewise, if $\tilde u$ is a minimizer of
\begin{equation}\label{mainmin1}
    \inf_{w\in BV_0(\Omega)}
    \int_\Omega
    \left(\ta|Dw+F|+H\,w\right),
\end{equation}
where $\ta$ is a perturbation of $a$, then the corresponding quantities $\tilde J$ and $\tilde \sigma$ can be defined analogously.

\begin{remark}
    Note that the primal problem \eqref{mainmin0} may not have a miimizer, but the dual problem (D) always has a minimizer. Throughout the paper we shall assume that \eqref{mainmin0} has a solution. For the existence of minimizers of weighted least gradient problems we refer to \cite{MR, MR2}
\end{remark}

\subsection{Stability of the vector field $J$}

    In this section, we show that the structure of the level sets of minimizers of the least gradient problem \eqref{mainmin0} is stable. The results generalize and are inspired by our earlier work in \cite{LM,MO}. Throughout the paper, we will assume that $a,\ta\in L^\infty(\Omega)$ and
    \begin{equation}\label{mbounds}
        0<m\leq a(x),\ta(x)\leq M,\,\,x\in\Omega,
    \end{equation}
    for some positive constants $m,M.$

    It is well known (see \cite{G,MB}) that there exists a constant $C_\Omega>0$, depending only on $\Omega$, such that
\[
\|u\|_{L^1(\Omega)}
\leq
C_\Omega \int_\Omega |Du|,
\qquad
\text{for all } u\in BV_0(\Omega).
\]
We will make use of the following lemmas.

\begin{lemma}\label{lem:a priori bound}
Assume that $F\in L^1(\Omega)$ satisfies
\[
\|F\|_{L^1(\Omega)}\leq k_1
\]
for some constant $k_1>0$, and suppose that
\begin{equation}\label{Hbound}
\|H\|_{L^\infty(\Omega)}
<
\frac{m}{C_\Omega}.
\end{equation}
Then there exists a constant $C$, depending only on $\Omega$, $m$, and $k_1$, such that
\begin{equation}\label{uppbound}
\max
\left\{
\|Du+F\|_{L^1(\Omega)},
\|D\tilde u+F\|_{L^1(\Omega)}
\right\}
\leq C,
\end{equation}
where $u$ and $\tilde u$ are minimizers of \eqref{mainmin0} and \eqref{mainmin1}, respectively. Moreover, the variational problem \eqref{mainmin1} is bounded from below.
\end{lemma}
{\bf Proof.} We have 
\allowdisplaybreaks
\begin{align*}
    \|Du\|_{L^1(\Omega)}-\|F\|_{L^1(\Omega)}
    &\leq\int_\Omega|Du+F|\,dx\nl
    &\leq\recip{m}\int_\Omega a|Du+F|\,dx\nl
    &=\recip{m}\int_\Omega a|Du+F|+Hu\,dx-\recip{m}\int_\Omega Hu\,dx\nl
    &\leq\frac{C_1}{m}+\recip{m}\|H\|_{L^\infty(\Omega)}\|u\|_{L^1(\Omega)}\nl
    &\leq\frac{C_1}{m}+\frac{C_\Omega}{m}\|H\|_{L^\infty(\Omega)}\|Du\|_{L^1(\Omega)}
\end{align*}
where $C_1=\left|\Int_{\Omega}a|Du+F|+Hu\,dx\right|.$ Note that $C_1$ is finite since
\begin{align*}
    Mk_1
    &\geq\int_{\Omega}a|F|\,dx\nl
    &=\Int_{\Omega}a|Dw+F|+Hw\,dx \bigg|_{w\equiv0}\nl
    &\geq\Int_{\Omega}a|Du+F|+Hu\,dx\nl
    &\geq \Int_{\Omega}m|Du|-m|F|+Hu\,dx\nl
    &\geq m\|Du\|_{L^1(\Omega)}-mk_1-\|H\|_{L^\infty(\Omega)}\|u\|_{L^1(\Omega)}\nl
    &\geq \left(m-C_\Omega\|H\|_{L^\infty(\Omega)}\right)\|Du\|_{L^1(\Omega)}-mk_1\nl
    &\geq -mk_1.
\end{align*}
 We conclude that
 $$
    \|Du\|_{L^1(\Omega)}
    \leq\frac{C_1+m\|F\|_{L^1(\Omega)}}{m-C_\Omega\|H\|_{L^\infty(\Omega)}}.
$$
Thus, $\|Du+F\|_{L^1(\Omega)}$ and $\|u\|_{L^1(\Omega)}$ are bounded above by a constant and, by a similar argument, so is $\|D\tu+F\|_{L^1(\Omega)},$ and the proof is complete. \hfill $\Box$

\begin{lemma}\label{lemma a1}
Under the assumptions of Lemma \ref{lem:a priori bound}, there exists a constant
\[
C=C(\Omega,m,k_1),
\]
independent of $u$ and $\tilde u$, such that
\[
\left|
\int_\Omega
\left(a|Du+F|+Hu\right)\,dx
-
\int_\Omega
\left(\tilde a|D\tilde u+F|+H\tilde u\right)\,dx
\right|
\leq
C\|a-\tilde a\|_{L^\infty(\Omega)}.
\]
\end{lemma}
{\bf Proof.}
Arguing as in the proof of Lemma 2.4 in \cite{MO},
\begin{align*}
    \Int_\Omega (a-\ta)|Du+F|\,dx
    &\leq\Int_\Omega a|Du+F|+Hu\,dx-\Int_\Omega \ta|D\tu+F|+H\tu\,dx\nl
    &\leq\Int_\Omega (a-\ta)|D\tu+F|\,dx
\end{align*}
and thus,
    \begin{align*}
    &\left|
        \Int_\Omega a|Du+F|-\ta|D\tu+F|+H(u-\tu)\,dx
    \right|\nl
    &\hspace{0.05in}\leq
    \max
    \left\{
        \|Du+F\|_{L^1(\Omega)},\|D\tu+F\|_{L^1(\Omega)}
    \right\}
    \|a-\ta\|_{L^\infty(\Omega)}.
\end{align*} \hfill $\Box$

Let $\nu_\Omega$ denote the outer unit normal vector to $\partial\Omega.$ Then for every $T\in\left(L^\infty(\Omega)\right)^n$ with $\nabla\cdot T\in L^n(\Omega),$ there exists a unique function $[T,\nu_\Omega]\in L^\infty(\partial\Omega)$ such that
\begin{equation}\label{wibp}
    \int_{\partial\Omega}[T,\nu_\Omega]u\,d\mathcal{H}^{n-1}=\int_\Omega u\nabla\cdot T\,dx+\int_\Omega T\cdot Du\,dx,\hspace{0.1in}u\in C^1(\bar\Omega).
\end{equation}
Moreover, for $u\in BV(\Omega)$ and $T\in\left(L^\infty(\Omega)\right)^n$ with $\nabla\cdot T\in L^n(\Omega),$ the linear functional $u\mapsto (T\cdot Du)$ gives rise to a Radon measure on $\Omega$ and \eqref{wibp} holds for all $u\in BV(\Omega)$ (see \cite{GA1,GA2}).

\begin{lemma}\label{mainlemma} Suppose that $0\leq\sigma(x)\leq\sigma_1$ for some constant $\sigma_1>0.$ Then
$$
    \int_\Omega |J||\tJ|-J\cdot\tJ\,dx\leq C\|a-\ta\|_{L^\infty(\Omega)},
$$
where $C=C(\Omega, m,k_1,\sigma_1)$ is a constant independent of $u$ and $\tu.$
\end{lemma}
{\bf Proof.}
Note that by the weak integration by parts formula,
\begin{align*}
    \int_\Omega u\nabla\cdot\tJ+\tilde J\cdot Du\,dx
    =\int_{\partial\Omega}[\tJ,\nu_\Omega]u\,dx
    =0
    =\int_{\partial\Omega}[\tJ,\nu_\Omega]\tu\,dx
    =\int_\Omega \tu\nabla\cdot\tJ+\tJ\cdot D\tu\,dx
\end{align*}
and thus,
$$
    \int_\Omega Hu+\tJ\cdot Du\,dx=\int_\Omega H\tu+\tJ\cdot D\tu\,dx.
$$
Similar to the proof of Lemma 2.5 in \cite{MO},
\allowdisplaybreaks
    \begin{align*}
        \int_\Omega |J||\tJ|-J\cdot\tJ\,dx
        &\leq\sigma_1\int_\Omega (\ta-a)|Du+F|\,dx\nl
        &\hspace{0.6in}+\sigma_1\int_{\Omega}a|Du+F|+Hu-\ta|D\tu+F|-H\tu\,dx\nl
        &\leq \sigma_1\big(\|Du+F\|_{L^1(\Omega)}+C\big)\|a-\ta\|_{L^\infty(\Omega)}
    \end{align*}
    where $C=C(\Omega,m,k_1)$ is the constant obtained in Lemma \ref{lemma a1}. By Lemma \ref{lem:a priori bound}, the proof is complete.\hfill $\Box$

Now we are ready to state the main result of this subsection. 
\begin{theorem}\label{thm a1}
    Suppose that $0\leq\sigma(x)\leq\sigma_1$ for some constant $\sigma_1>0.$ Then
$$
    \|J-\tJ\|_{L^1(\Omega)}\leq C\|a-\ta\|_{L^\infty(\Omega)}^{\frac{1}{2}},
$$
where $C=C(\Omega, m,k_1,\sigma_1)$ is a constant independent of $u$ and $\tu.$
\end{theorem}
{\bf Proof.}
    The proof follows from Lemma \ref{mainlemma} together with an argument similar to that used in the proof of Theorem 2.6 in \cite{MO}.
\hfill $\Box$

\subsection{$L^1$ stability of the minimizers}

In this section, we prove the stability of minimizers of the least gradient problem \eqref{mainmin0} in $L^1(\Omega)$. Since minimizers of \eqref{mainmin0} are not unique in general, additional assumptions on the weights $a$, $\tilde a$, and on the associated minimizers are required in order to obtain stability results.

\begin{definition}
Let $\sigma_0,\sigma_1>0$ be fixed constants. We say that
\[
u\in C^1(\overline{\Omega})
\]
is \emph{admissible} if $u$ is a minimizer of \eqref{mainmin0} and there exists a function $\sigma\in C(\Omega)$ satisfying
\[
0<\sigma_0\leq \sigma \leq \sigma_1,
\]
such that
\[
J=\sigma(\nabla u+F)
\]
and
\[
m\leq |J|\leq M,
\]
where the positive constants $m$ and $M$ are given in \eqref{mbounds}. The notion of admissibility for $\tilde u$ is defined analogously.
\end{definition}

As in \cite{MO}, for the remainder of the paper, we will assume that $F$ is a conservative vector field $\big($i.e. $F=\nabla f$ for some function $f\in C^1(\Omega)\big).$ Furthermore, assume $f\in C(\bar\Omega)$ and on $\bar\Omega,$ define $v:=u+f$ and $\tv:=\tu+f.$ Then $v-\tv=u-\tu.$

  We will prove our results in dimension $n=2$ and then extend them to dimension $n=3.$ Let $v\in C^1(\bar\Omega)$ with $|\nabla v|>0$ almost everywhere in $\Omega.$ By the regularity result of De Giorgi, (Theorem 4.11 in \cite{EG}) it follows that almost all level sets of $v$ are $C^1$ hypersurfaces. For $n=2,$ we will furthermore assume that the length of level sets of $v$ in $\Omega$ is uniformly bounded, i.e.
  \begin{equation}\label{crvbnd}
      \sup\limits_{t\inR}\int_{\{v=t\}\cap\Omega}\,dl=K<\infty.
  \end{equation}

The following result follows from an argument along the same lines as that of Theorem 3.3 in \cite{MO}.

  \begin{theorem}\label{thm a2}
      Let $n=2,$ and suppose $u$ and $\tu$ are admissible with $u|_{\partial\Omega}=0=\tu|_{\partial\Omega}.$ If $v$ satisfies \eqref{crvbnd} then
    $$
        \|u-\tu\|_{L^1(\Omega)}\leq C\|a-\ta\|_{L^\infty(\Omega)}^{\frac{1}{2}},
    $$
    for some constant $C(\Omega, m,M,k_1,K,\sigma_0,\sigma_1)$ independent of $\tu$ and $\tsigma.$
  \end{theorem}

The following result follows from an argument similar to that used in the proof of Theorem 3.5 in \cite{MO}.

\begin{theorem}\label{thm a3}
    Let $n=3$ and suppose that $u$ and $\tu$ are admissible with $u|_{\partial\Omega}=\tu|_{\partial\Omega}=0.$ Suppose that the level sets of $v$ can be foliated to one-dimensional curves as in Definition 3.4 in \cite{MO}. Then
$$
    \|u-\tu\|_{L^1(\Omega)}\leq C\|a-\ta\|_{L^\infty(\Omega)}^{1/2}
$$
for some constant $C(\Omega, m,M,k_1,K,\sigma_0,\sigma_1,c_g,C_g)$ independent of $\tu$ and $\tsigma.$
\end{theorem}

\subsection{$W^{1,1}$ stability of the minimizers}

In this section, we prove the stability of the minimizers of \eqref{mainmin0} in $W^{1,1}.$ As mentioned in Section 2.2, in general, \eqref{mainmin0} does not have unique minimizers, so to prove stability results, it is natural to expect stronger assumptions on the minimizers.

The following two stability result can be established by an argument similar to that used in the proof of Theorem 4.3 and 4.4 in \cite{MO}, and we omit the details for brevity.

\begin{theorem}\label{thm a4}
    Let $n=2$ and suppose that $u$ and $\tu$ are admissible with $u|_{\partial\Omega}=\tu|_{\partial\Omega}=0.$ Suppose $F\in C^2(\bar\Omega),$ $H\in C^1(\bar\Omega),$ $\sigma,\tsigma\in C^2(\bar\Omega)$ and
    \begin{equation}\label{FHineq}
    \|F\|_{C^2(\Omega)},\|H\|_{C^1(\Omega)}\leq k_2
\end{equation}
and
\begin{equation}\label{sigmaineq}
    \|\sigma\|_{C^2(\Omega)},\|\tsigma\|_{C^2(\Omega)}\leq\sigma_2
\end{equation}
    for some $k_2,\sigma_2>0.$ If $v$ satisfies \eqref{crvbnd} and the level sets of $v$ are well-structured in the sense of Definition 4.2 in \cite{MO}, then
\begin{equation}
    \|\nabla u-\nabla\tu\|_{L^1(\Omega)}\leq C\|\,a-\ta\|_{L^\infty(\Omega)}^{1/4},
\end{equation}
for some constant $C=C(\Omega,m,M,k_1,k_2,K,\tK,\sigma_0,\sigma_1,\sigma_2)$ independent of $\tu$ and $\tsigma.$
\end{theorem}

\begin{theorem}\label{thm a5}
    Let $n=3$ and suppose that $u$ and $\tu$ are admissible with $u|_{\partial\Omega}=\tu|_{\partial\Omega}=0.$ Suppose $F\in C^2(\bar\Omega),$ $H\in C^1(\bar\Omega),$ $\sigma,\tsigma\in C^2(\bar\Omega)$ and satisfy \eqref{FHineq} and \eqref{sigmaineq}. If $v$ satisfies \eqref{crvbnd}, the level sets of $v$ can be foliated to one-dimensional curves in the sense of Definition 3.4 in \cite{MO}, and the level sets of $v$ are well-structured in the sense of Definition 4.2 in \cite{MO}, then
\begin{equation}
    \|\nabla u-\nabla\tu\|_{L^1(\Omega)}\leq C\|\,a-\ta\|_{L^\infty(\Omega)}^{1/4},
\end{equation}
for some constant $C=C(\Omega,m,M,k_1,k_2,K,\tK,\sigma_0,\sigma_1,\sigma_2)$ independent of $\tu$ and $\tsigma.$
\end{theorem}

\begin{theorem}\label{thm a6}
Let $n=2$ and suppose that $u$ and $\tu$ are admissible with $u|_{\partial\Omega}=\tu|_{\partial\Omega}=0.$ Suppose $F\in C^2(\bar\Omega),$ $H\in C^1(\bar\Omega),$ $\sigma,\tsigma\in C^2(\bar\Omega)$ and satisfy \eqref{FHineq} and \eqref{sigmaineq}. If $v$ satisfies \eqref{crvbnd} and the level sets of $v$ are well-structured in the sense of Definition 4.2 in \cite{MO}, then
\begin{equation}
    \|\sigma-\tsigma\|_{L^1(\Omega)}\leq C\|\,a-\ta\|_{L^\infty(\Omega)}^{1/4},
\end{equation}
for some constant $C=C(\Omega,m,M,k_1,k_2,K,\tK,\sigma_0,\sigma_1,\sigma_2)$ independent of $\tu$ and $\tsigma.$
\end{theorem}
{\bf Proof.}
    The proof follows from Theorem \ref{thm a4} and proceeds by adapting the argument used in the proof of Theorem 4.6 in \cite{LM2}, with $u$ and $\tu$ replaced by $v$ and $\tv$, respectively.
\hfill$\Box$

Similarly, the above theorem continues to hold in dimension $n=3$, as a consequence of Theorem \ref{thm a5} combined with arguments and calculations analogous to those appearing in the proof of Theorem 4.6 in \cite{LM2}.

\begin{theorem}\label{thm a7}
Let $n=3$ and suppose that $u$ and $\tu$ are admissible with $u|_{\partial\Omega}=\tu|_{\partial\Omega}=0.$ Suppose $F\in C^2(\bar\Omega),$ $H\in C^1(\bar\Omega),$ $\sigma,\tsigma\in C^2(\bar\Omega)$ and satisfy \eqref{FHineq} and \eqref{sigmaineq}. If $v$ satisfies \eqref{crvbnd}, the level sets of $v$ can be foliated to one-dimensional curves in the sense of Definition 3.4 in \cite{MO}, and the level sets of $v$ are well-structured in the sense of Definition 4.2 in \cite{MO}, then
\begin{equation}
    \|\sigma-\tsigma\|_{L^1(\Omega)}\leq C\|\,a-\ta\|_{L^\infty(\Omega)}^{1/4},
\end{equation}
for some constant $C=C(\Omega,m,M,k_1,k_2,K,\tK,\sigma_0,\sigma_1,\sigma_2)$ independent of $\tu$ and $\tsigma.$
\end{theorem}

\section{Stability with respect to the vector field $F$}
In this section, we establish stability results for minimizers of \eqref{mainmin0} with respect to perturbations of the vector field $F$. More precisely, we study how minimizers vary under suitable perturbations of the drift term in the functional and obtain quantitative estimates describing the dependence of minimizers on the underlying vector field. Let us start by proving the following lemma.

\begin{lemma}\label{lemma F1} The following inequality holds: 
$$
    \left|
        \int_\Omega a|Du+F|+Hu\,dx
        -\int_\Omega a|D\tu+\tF|+H\tu\,dx
    \right|
    \leq M\|F-\tF\|_{L^1(\Omega)}.
$$
\end{lemma}
{\bf Proof.}
\begin{align*}
    -M\|F-\tF\|_{L^1(\Omega)}
    &\leq\Int_\Omega -a|F-\tF|\,dx\nl
    &\leq\Int_\Omega a\big(|Du+F|-|Du+\tF|\big)\,dx\nl
    &=\Int_\Omega a|Du+F|+Hu\,dx-\Int_\Omega a|Du+\tF|+H u\,dx\nl
    &\leq\Int_\Omega a|Du+F|+Hu\,dx-\Int_\Omega a|D\tu+\tF|+H\tu\,dx\nl
    &\leq\Int_\Omega a|D\tu+F|+H\tu\,dx-\Int_\Omega a|D\tu+\tF|+H\tu\,dx\nl
    &=\Int_\Omega a\big(|D\tu+F|-|D\tu+\tF|\big)\,dx\nl
    &\leq\Int_\Omega a|F-\tF|\,dx\nl
    &\leq M\|F-\tF\|_{L^1(\Omega)}.
\end{align*}
Thus,
$$
    \left|
        \Int_\Omega a|Du+F|+Hu\,dx-\Int_\Omega a|D\tu+F|+\tH\tu\,dx
    \right|\nl
    \leq M\|F-\tF\|_{L^1(\Omega)}.
$$
\hfill $\Box$

\begin{lemma}\label{lemma F2} Suppose that $0\leq\sigma(x)\leq\sigma_1$ for some constant $\sigma_1>0.$ Then
$$
    \int_\Omega |J||\tJ|-J\cdot\tJ\,dx\leq 2M\sigma_1\|F-\tF\|_{L^1(\Omega)}.
$$
\end{lemma}
{\bf Proof.}
Similar to the proof of Lemma 2.5 in \cite{MO},
\allowdisplaybreaks
    \begin{align*}
        \int_\Omega |J||\tJ|-J\cdot\tJ\,dx
        &\leq\sigma_1\int_\Omega a|Du+F|+Hu-a|D\tu+F|-H\tu\,dx\nl
        &\leq\sigma_1\int_\Omega a|Du+F|+Hu-a|D\tu+\tF|-H\tu+a|F-\tF|\,dx\nl
        &\leq 2M\sigma_1\|F-\tF\|_{L^1(\Omega)}
    \end{align*}
    where we have used Lemma \ref{lemma F1}.
\hfill $\Box$

The proof of the following theorem is based on an adaptation of the argument used in Theorem 2.6 of \cite{MO}.

\begin{theorem}\label{thm F1}
    Suppose that $0\leq\sigma(x)\leq\sigma_1$ for some constant $\sigma_1>0.$ Then
$$
    \|J-\tJ\|_{L^1(\Omega)}\leq (4M\sigma_1|\Omega|)^{1/2}\|F-\tF\|_{L^1(\Omega)}^{\frac{1}{2}}.
$$
\end{theorem}

\subsection{$L^1$ stability of the minimizers}

Similar to Section 2.2, we will assume that $F$ and $\tF$ are conservative vector
fields (i.e. $F=\nabla f$ and $\tF=\nabla\tf$ for some functions $f,\tf\in C^1(\Omega)\cap C(\bar\Omega).$ On $\bar\Omega,$ define $v:=u+f$ and $\tv:=\tu+\tf.$

\begin{theorem}\label{thm F2}
      Let $n=2,$ and suppose $u$ and $\tu$ are admissible with $u|_{\partial\Omega}=0=\tu|_{\partial\Omega}.$ If $v$ satisfies \eqref{crvbnd} then
    $$
        \|u-\tu\|_{L^1(\Omega)}\leq 
        C\cdot\max\left( \|f-\tf\|_{W^{1,1}(\Omega)}^{\frac{1}{2}},\,\|f-\tf\|_{W^{1,1}(\Omega)}\right)
    $$
    for some constant $C=C(\Omega,m, M,K,\sigma_0,\sigma_1)$ independent of $f,\,\tf,F,\text{ and }\tF.$
\end{theorem}
{\bf Proof.}
    The argument follows the same strategy as the proof of Theorem 2.10 in \cite{MO}. Since $u$ is admissible,
    $$
        0<\frac{m}{\sigma_1}\leq|\nabla u+F|=|\nabla v|
    $$
    for all $x\in\Omega.$ From the coarea formula, it follows that
    \begin{equation}\label{i2}
        \frac{m}{\sigma_1}\int_\Omega|u-\tu|-|f-\tf|\,dx\leq
        \frac{m}{\sigma_1}\int_\Omega|v-\tv|\,dx\leq\int_\Omega|\nabla v||v-\tv|\,dx=\int_\R\int_{\{v=t\}\cap\Omega}|v-\tv|\,dl\,dt.
    \end{equation}
    Let $\Gamma_t$ be a connected component of $\{x\in\Omega:v(x)=t\}$ and let $\gamma:[0,L]\to\Gamma_t$ be a path parameterized by the arc length of $\Gamma_t$ with $\gamma(0)\in\partial\Omega.$ Now define $h:[0,L]\to\R$ by
    $$
        h(s)=u(\gamma(s))-\tu(\gamma(s)).
    $$
    Then $h(0)=0.$ Since $v$ is constant on $\Gamma_t,$ then $0=\frac{d}{ds}v(\gamma(s))=\nabla v(\gamma(s))\cdot\gamma'(s)$ on $\Gamma_t.$ Therefore,
    \begin{align*}
        h'(s)
        &=\nabla u(\gamma(s))\cdot\gamma'(s)-\nabla \tu(\gamma(s))\cdot\gamma'(s)\nl
        &=\nabla v(\gamma(s))\cdot\gamma'(s)-\nabla \tv(\gamma(s))\cdot\gamma'(s)-\big(F(\gamma(s))-\tF(\gamma(s))\big)\cdot\gamma'(s)\nl
        &=\frac{\sigma(\gamma(s))}{\tsigma(\gamma(s))}\nabla v(\gamma(s))\cdot\gamma'(s)-\nabla \tv(\gamma(s))\cdot\gamma'(s)-\big(F(\gamma(s))-\tF(\gamma(s))\big)\cdot\gamma'(s)\nl
        &=\frac{J(\gamma(s))-\tJ(\gamma(s))}{\tsigma(\gamma(s))}\cdot\gamma'(s)-\big(F(\gamma(s))-\tF(\gamma(s))\big)\cdot\gamma'(s).
    \end{align*}
    Note that there exists $x_t^*$ on $\Gamma_t$ such that
    $$
        |u(x_t^*)-\tu(x_t^*)|=\max\limits_{x\in\Gamma_t}|u(x)-\tu(x)|.
    $$
    Furthermore, there exists $s_0\in[0,L]$ such that $x_t^*=\gamma(s_0).$ Then
    \begin{align*}
        |u(x_t^*)-\tu(x_t^*)|
        &=|h(s_0)|\nl
        &=\left|\int_0^{s_0}h'(r)\,dr\right|\nl
        &\leq\recip{\sigma_0}\int_0^L|J(\gamma(r))-\tJ(\gamma(r))|\,dr+\int_0^L|F(\gamma(r))-\tF(\gamma(r))|\,dr\nl
        &=\recip{\sigma_0}\int_{\Gamma_t}|J-\tJ|\,dl+\int_{\Gamma_t}|F-\tF|\,dl.
    \end{align*}
    Thus,
    \begin{align*}
        \int_{\Gamma_t}|v-\tv|\,dl
        &\leq\int_{\Gamma_t}|u-\tu|+|f-\tf|\,dl\nl
        &\leq |u(x_t^*)-\tu(x_t^*)|\int_{\Gamma_t}\,dl+\int_{\Gamma_t}|f-\tf|\,dl\nl
        &\leq K\left(\recip{\sigma_0}\int_{\Gamma_t}|J-\tJ|\,dl+\int_{\Gamma_t}|F-\tF|\,dl\right)+\int_{\Gamma_t}|f-\tf|\,dl
    \end{align*}
    and therefore,
    $$
        \int_{\{v=t\}\cap\Omega}|v-\tv|\,dl
        \leq\frac{K}{\sigma_0}\int_{\{v=t\}\cap\Omega}|J-\tJ|\,dl+\max(K,1)\int_{\{v=t\}\cap\Omega}|F-\tF|+|f-\tf|\,dl.
    $$
    By the coarea formula and Theorem \ref{thm F1},
    \begin{align*}
        &\int_\R\int_{\{v=t\}\cap\Omega}|v-\tv|\,dl\,dt\nl
        &\leq\frac{K}{\sigma_0}\int_\R\int_{\{v=t\}\cap\Omega}|J-\tJ|\,dl\,dt+\max(K,1)\Int_\R\int_{\{v=t\}\cap\Omega}|F-\tF|+|f-\tf|\,dl\,dt\nl
        &=\frac{K}{\sigma_0}\int_\Omega|\nabla v||J-\tJ|\,dx+\max(K,1)\Int_{\Omega}|\nabla v|\big(|F-\tF|+|f-\tf|\big)\,dx\nl
        &\leq\frac{KM}{\sigma_0^2}\int_\Omega|J-\tJ|\,dx+\frac{\max(K,1)\cdot M}{\sigma_0}\Int_{\Omega}|F-\tF|+|f-\tf|\,dx\nl
        &\leq\frac{KM(4M\sigma_1|\Omega|)^{1/2}}{\sigma_0^2}\|F-\tF\|_{L^1(\Omega)}^{\frac{1}{2}}+\frac{\max(K,1)\cdot M}{\sigma_0}\|f-\tf\|_{W^{1,1}(\Omega)}\nl
        &\leq C \left( \|f-\tf\|_{W^{1,1}(\Omega)}^{\frac{1}{2}} +\|f-\tf\|_{W^{1,1}(\Omega)}\right).
    \end{align*}
    for some constant $C=C(\Omega,M,K,\sigma_0,\sigma_1).$ By \eqref{i2}, the proof is complete.
    \hfill$\Box$

\begin{theorem}\label{thm F3}
    Let $n=3$ and suppose that $u$ and $\tu$ are admissible with $u|_{\partial\Omega}=\tu|_{\partial\Omega}=0.$ Suppose that the level sets of $v$ can be foliated to one-dimensional curves as in Definition 3.4 in \cite{MO}. Then
$$
    \|u-\tu\|_{L^1(\Omega)}\leq C\cdot\max\left( \|f-\tf\|_{W^{1,1}(\Omega)}^{\frac{1}{2}},\,\|f-\tf\|_{W^{1,1}(\Omega)}\right)
$$
for some constant $C(\Omega, m,M,K,\sigma_0,\sigma_1,c_g,C_g)$ independent of $\tu$ and $\tsigma.$
\end{theorem}
{\bf Proof.}
    The proof proceeds along the same lines as that of Theorem \ref{thm F2}. Since $u$ is admissible, then
    \begin{equation}\label{II1}
        \frac{m}{\sigma_1}\int_\Omega|u-\tu|-|f-\tf|\,dx\leq\frac{m}{\sigma_1}\int_\Omega|v-\tv|\,dx
        \leq\int_\Omega|\nabla v||v-\tv|\,dx
        =\int_\R\int_{\{v=t\}\cap\Omega}|v-\tv|\,dS\,dt
    \end{equation}
    by the coarea formula. Consider $g_t$ from Definition 3.4 in \cite{MO}. By the coarea formula, 
    \begin{align}\label{II2}
    \begin{split}
        \int_\R\int_{\{v=t\}\cap\Omega}|v-\tv|\,dS\,dt
        &=\int_\R\int_\R\int_{\{v=t\}\cap\{g_t=r\}\cap\Omega}\recip{|\nabla g_t|}|v-\tv|\,dl\,dr\,dt\nl
        &\leq\recip{c_g}\int_\R\int_\R\int_{\{v=t\}\cap\{g_t=r\}\cap\Omega}|v-\tv|\,dl\,dr\,dt.
    \end{split}
    \end{align}
    Similar to the proof of Theorem \ref{thm F2},
    \begin{align}\label{II3}
    \begin{split}
        &\int_{\{v=t\}\cap\{g_t=r\}\cap\Omega}|v-\tv|\,dl\nl
        &\leq K\left(\recip{\sigma_0}\int_{\{v=t\}\cap\{g_t=r\}\cap\Omega}|J-\tJ|\,dl+\int_{\{v=t\}\cap\{g_t=r\}\cap\Omega}|F-\tF|\,dl\right)+\int_{\{v=t\}\cap\{g_t=r\}\cap\Omega}|f-\tf|\,dl.
    \end{split}
    \end{align}
    By the coarea formula and Theorem \ref{thm F1},
    \begin{align}\label{II4}
    \begin{split}
        \int_\R\int_\R\int_{\{v=t\}\cap\{g_t=r\}\cap\Omega}|J-\tJ|\,dl\,dr\,dt
        &=\int_\R\int_{\{v=t\}\cap\Omega}|\nabla g_t||J-\tJ|\,dS\,dt\\
        &\leq C_g\int_\R\int_{\{v=t\}\cap\Omega}|J-\tJ|\,dS\,dt\\
        &=C_g\int_\Omega|\nabla v||J-\tJ|\,dx\\
        &\leq\frac{C_g M}{\sigma_0}\int_\Omega|J-\tJ|\,dx\\
        &\leq\frac{C_g M}{\sigma_0}(4M\sigma_1|\Omega|)^{1/2}\|F-\tF\|_{L^1(\Omega)}^{\frac{1}{2}}.
    \end{split}
    \end{align}
    Similarly,
    \begin{equation}\label{II5}
        \int_\R\int_\R\int_{\{v=t\}\cap\{g_t=r\}\cap\Omega}|F-\tF|\,dl\,dr\,dt
        \leq\frac{C_g M}{\sigma_0}\|F-\tF\|_{L^1(\Omega)}
    \end{equation}
    and
    \begin{equation}\label{II6}
        \int_\R\int_\R\int_{\{v=t\}\cap\{g_t=r\}\cap\Omega}|f-\tf|\,dl\,dr\,dt
        \leq\frac{C_g M}{\sigma_0}\|f-\tf\|_{L^1(\Omega)}
    \end{equation}
    By \eqref{II1} through \eqref{II6}, the proof is complete.
    \hfill$\Box$

\subsection{$W^{1,1}$ stability of the minimizers}

In this section, we prove the stability of the minimizers of \eqref{mainmin0} in $W^{1,1}$ with respect to $F$. 

\begin{lemma}\label{lemma F3}
    Let $n=2,3,$ and suppose that $u,\tu\in H_0^1(\Omega)$ are admissible. Suppose $\partial\Omega$ is of class $C^3,$ $F,\tF\in C^2(\bar\Omega),$ $H\in C^1(\bar\Omega),$ and $\sigma,\tsigma\in C^{2,1}(\bar\Omega)$ 
\begin{equation}\label{FHineq3}
    \|F\|_{C^2(\Omega)},\|\tF\|_{C^2(\Omega)},\|H\|_{C^1(\Omega)}\leq k_2
\end{equation}
and
\begin{equation}\label{sigmaineq3}
    \|\sigma\|_{C^2(\Omega)},\|\tsigma\|_{C^2(\Omega)}\leq\sigma_2
\end{equation}
for some $k_2,\sigma_2>0.$ Let
\begin{equation}\label{G2}
    G(x)=\frac{\tJ(x)-J(x)}{\tsigma(x)}+F(x)-\tF(x),\,x\in\Omega,
\end{equation}
with $G=(G_1,G_2)$ for $n=2$ and $G=(G_1,G_2,G_3)$ for $n=3.$ Then
$$
    \|\nabla G_i\|_{L^1(\Omega)}\leq C\big(\|J-\tJ\|_{L^1(\Omega)}^{1/2}+\|F-\tF\|_{L^1(\Omega)}^{1/2}\big),\hspace{0.1in}1\leq i\leq n
$$
for some constant $C\left(\Omega,M,k_2,\sigma_0,\sigma_2\right).$
\end{lemma}
{\bf Proof.}
Similar to the proof of Lemma 2.14 in \cite{MO},
    \begin{equation}\label{k1}
        \|\nabla G_i\|_{L^1(\Omega)}\leq C_1\left(\|G_i\|_{L^1(\Omega)}^{1/2}+\|G_i\|_{L^1(\Omega)}\right)
    \end{equation}
    for some $C_1$ depending on $k_2, \sigma_0,\sigma_2,$ and $\Omega.$ Combining \eqref{k1} with
    \begin{align*}
        \|G_i\|_{L^1(\Omega)}
        &\leq\frac{\|J-\tJ\|_{L^1(\Omega)}}{\sigma_0}+\|F-\tF\|_{L^1(\Omega)}\nl
        &\leq\frac{(2M|\Omega|)^{1/2}\|J-\tJ\|_{L^1(\Omega)}^{1/2}}{\sigma_0}+(2k_2|\Omega|)^{1/2}\|F-\tF\|_{L^1(\Omega)}^{1/2},
    \end{align*}
    we obtain the desired result.
\hfill$\Box$

\begin{theorem}\label{thm F4}
    Let $n=2$ and suppose that $u$ and $\tu$ are admissible with $u|_{\partial\Omega}=\tu|_{\partial\Omega}=0.$ Suppose $F,\tF\in C^2(\bar\Omega),\,H\in C^1(\bar\Omega),\,\sigma,\tsigma\in C^2(\bar\Omega)$ and satisfy \eqref{FHineq3} and \eqref{sigmaineq3}. If the level sets of $v$ are well-structured in the sense of Definition 4.2 in \cite{MO}, then
\begin{equation}
    \|\nabla u-\nabla\tu\|_{L^1(\Omega)}\leq C\|F-\tF\|_{L^1(\Omega)}^{1/4},
\end{equation}
for some constant $C=C(\Omega,m,M,k_2,K,\tK,\sigma_0,\sigma_1,\sigma_2)$ independent of $\tu$ and $\tsigma.$
\end{theorem}
{\bf Proof.}
    Fix $x\in\Omega$ and $h\in S^1$ and define
    $$
        \CL(x,h)
        :=(\nabla\tu(x)-\nabla u(x))\cdot h
        =\lim\limits_{t\to0}\Frac{[\tu(x+th)-u(x+th)]-[\tu(x)-u(x)]}{t}.
    $$
    Since the level sets of $v$ reach the boundary $\partial\Omega,$ there exist $y,y_t\in\partial\Omega$ such that
    $$
        u(x)=v(x)-f(x)=v(y)-f(x)=u(y)+f(y)-f(x)=\tu(y)+f(y)-f(x)=f(y)-f(x),
    $$
    and
    \begin{align*}
        u(x+th)=v(x+th)-f(x+th)=v(y_t)-f(x+th)
        &=u(y_t)+f(y_t)-f(x+th)\nl
        &=\tu(y_t)+f(y_t)-f(x+th)\nl
        &=f(y_t)-f(x+th).
    \end{align*}
    Thus,
    \begin{align*}
        &[\tu(x+th)-u(x+th)]-[\tu(x)-u(x)]\nl
        &=[\tu(x+th)-\tu(y_t)]-[\tu(x)-\tu(y)]+[f(x+th)-f(x)]-[f(y_t)-f(y)].
    \end{align*}
    Consider $\gamma$ and $\gamma_t,$ curves passing through $x$ and $x+th,$ as in Definition 4.2 in \cite{MO} with $\gamma(0)=y$ and $\gamma_t(0)=y_t.$ Suppose $\gamma(s_0)=x$ and $\gamma_t(s_0)=x+th$ (where $\gamma_t$ is reparametrized if necessary). Then
    \begin{align*}
        &[\tu(x+th)-u(x+th)]-[\tu(x)-u(x)]\nl
        &=[\tu(x+th)-\tu(y_t)]-[\tu(x)-\tu(y)]+[f(x+th)-f(y_t)+f(x)-f(y)]\nl
        &=[\tu(\gamma_t(s_0))-\tu(\gamma_t(0))]-[\tu(\gamma(s_0))-\tu(\gamma(0))]+[f(\gamma_t(s_0))-f(\gamma_t(0))+f(\gamma(s_0))-f(\gamma(0))]\nl
        &=\int_0^{s_0}\nabla\tu(\gamma_t(s))\cdot\gamma_t'(s)\,ds-\int_0^{s_0}\nabla\tu(\gamma(s))\cdot\gamma'(s)\,ds\nl
        &\hspace{0.2in}+\int_0^{s_0}F(\gamma_t(s))\cdot\gamma_t'(s)\,ds-\int_0^{s_0}F(\gamma(s))\cdot\gamma'(s)\,ds\nl
        &=\int_0^{s_0}\nabla\tv(\gamma_t(s))\cdot\gamma_t'(s)\,ds-\int_0^{s_0}\nabla\tv(\gamma(s))\cdot\gamma'(s)\,ds\nl
        &\hspace{0.2in}+\int_0^{s_0}\big(F(\gamma_t(s))-\tF(\gamma_t(s))\big)\cdot\gamma_t'(s)\,ds-\int_0^{s_0}\big(F(\gamma(s))-\tF(\gamma(s))\big)\cdot\gamma'(s)\,ds.
    \end{align*}
    Thus,
    \begin{align*}
        &\CL(x,h)=\lim\limits_{t\to0}\recip{t}\bigg(\int_0^{s_0}\big(\nabla\tv(\gamma_t(s))+F(\gamma_t(s))-\tF(\gamma_t(s))\big)\cdot\gamma_t'(s)\,ds\nl
        &\hspace{1.5in}-\int_0^{s_0}\big(\nabla\tv(\gamma(s))+F(\gamma(s))-\tF(\gamma(s))\big)\cdot\gamma'(s)\,ds\bigg).
    \end{align*}
    On $\Gamma,$ we have $\frac{d}{ds}(v(\gamma(s)))
    =\nabla v(\gamma(s))\cdot\gamma'(s)
    =0
    =\frac{\sigma(\gamma(s))}{\tsigma(\gamma(s))}\nabla v(\gamma(s))\cdot\gamma'(s)
    =\frac{J(\gamma(s))}{\tsigma(\gamma(s))}\cdot\gamma'(s).$ Thus, $\frac{J(\gamma(s))}{\tsigma(\gamma(s))}\cdot\gamma'(s)=0.$ Similarly, on $\Gamma_t,$ we have
    $
        \frac{J(\gamma_t(s))}{\tsigma(\gamma_t(s))}\cdot\gamma_t'(s)=0.
    $
    Therefore,
    \begin{align*}
        \CL(x,h)
        &=\lim\limits_{t\to0}\recip{t}\Bigg(\int_0^{s_0}\Bigg(\frac{\tJ(\gamma_t(s))}{\tsigma(\gamma_t(s))}+F(\gamma_t(s))-\tF(\gamma_t(s))\Bigg)\cdot\gamma_t'(s)\,ds\nl
        &\hspace{1in}-\int_0^{s_0}\Bigg(\frac{\tJ(\gamma(s))}{\tsigma(\gamma(s))}+F(\gamma(s))-\tF(\gamma(s))\Bigg)\cdot\gamma'(s)\,ds\Bigg)\nl
        &=\lim\limits_{t\to0}\recip{t}\Bigg(\int_0^{s_0}\Bigg(\frac{\tJ(\gamma_t(s))-J(\gamma_t(s))}{\tsigma(\gamma_t(s))}+F(\gamma_t(s))-\tF(\gamma_t(s))\Bigg)\cdot\gamma_t'(s)\,ds\nl
        &\hspace{1in}-\int_0^{s_0}\Bigg(\frac{\tJ(\gamma(s))-J(\gamma(s))}{\tsigma(\gamma(s))}+F(\gamma(s))-\tF(\gamma(s))\Bigg)\cdot\gamma'(s)\,ds\Bigg)
    \end{align*}
    Define
    $$
        G(x):=\frac{\tJ(x)-J(x)}{\tsigma(x)}+F(x)-\tF(x),\,x\in\Omega.
    $$
    Then
    $$
        \CL(x,h)
        =\lim\limits_{t\to0}\,\recip{t}\left(\int_0^{s_0}G(\gamma_t(s))\cdot\gamma_t'(s)\,ds-\int_0^{s_0}G(\gamma(s))\cdot\gamma'(s)\,ds\right)=\lim\limits_{t\to0}\mathcal{G}(t)
    $$
    where
    \begin{equation}\label{CG2}
        \mathcal{G}(t):=\recip{t}\int_0^{s_0}\left(G(\gamma_t(s))-G(\gamma(s))\right)\cdot\gamma_t'(s)\,ds
        +\recip{t}\int_0^{s_0}G(\gamma(s))\cdot(\gamma_t'(s)-\gamma'(s))\,ds.
    \end{equation}
    By Definition 4.2 in \cite{MO}, there exists a positive constant $\tK$ such that $\left|\Frac{\gamma_t'(s)-\gamma'(s)}{t}\right|\leq \tK.$ Thus, we obtain the following bound for the second integral in \eqref{CG2}:
    \begin{equation}\label{estG2F}
        \left|\recip{t}\int_0^{s_0}G(\gamma(s))\cdot(\gamma_t'(s)-\gamma'(s))\,ds\right|\leq\int_0^L\frac{\tK}{\sigma_0}|\tJ(\gamma(s))-J(\gamma(s))|+\tK|F(\gamma(s))-\tF(\gamma(s))|\,ds.
    \end{equation}
    Next, by Definition 4.2 in \cite{MO},
    $$
        \lim\limits_{t\to0}\frac{\gamma_t(s)-\gamma(s)}{t}=B_{x,h}(s).
    $$
    If $G=(G_1,G_2),$ then for $i\in\{1,2\},$
    $$
        \lim\limits_{t\to0}\frac{G_i(\gamma_t(s))-G_i(\gamma(s))}{t}
        =\lim\limits_{t\to0}\frac{G_i(\gamma(s)+tB_{x,h}(s))-G_i(\gamma(s))}{t}
        =\nabla G_i(\gamma(s))\cdot B_{x,h}(s).
    $$
    Thus, we estimate the first integral in \eqref{CG2} by
    \begin{align}
        &\lim\limits_{t\to 0}\recip{t}\int_0^{s_0}\big(G(\gamma_t(s))-G(\gamma(s))\big)\cdot\gamma_t'(s)\,ds\nonumber\nl
        &=\int_0^{s_0}\big(\nabla G_1(\gamma(s))\cdot B_{x,h}(s), \nabla G_2(\gamma(s))\cdot B_{x,h}(s)\big)\cdot\gamma_t'(s)\,ds\nonumber\nl
        &\leq\|B\|_{L^\infty(\Omega)}\int_0^{s_0}\big(|\nabla G_1(\gamma(s)|+|\nabla G_2(\gamma(s)|\big)\,|\gamma'(s)|\,ds\nonumber\nl
        &\leq\|B\|_{L^\infty(\Omega)}\int_0^L\big(|\nabla G_1(\gamma(s)|+|\nabla G_2(\gamma(s)|\big)\,ds\label{estG1F}
    \end{align}
    By \eqref{estG2F} and \eqref{estG1F}, we conclude that 
    \begin{align*}
        |\nabla\tu(x)-\nabla u(x)|
        \leq\sup\limits_{h\in S^1}\CL(x,h)
        &\leq\int_0^L\frac{\tK}{\sigma_0}|\tJ(\gamma(s))-J(\gamma(s))|+\tK|F(\gamma(s))-\tF(\gamma(s))|\,ds\nl
        &\hspace{0.6in}+\|B\|_{L^\infty(\Omega)}\int_0^L\big(|\nabla G_1(\gamma(s))|+|\nabla G_2(\gamma(s))|\big)\,ds.
    \end{align*}
    Thus,
    $$
        \int_\Gamma|\nabla\tu-\nabla u|\,dl
        \leq K\tK\int_\Gamma\recip{\sigma_0}|J-\tJ|+|F-\tF|\,dl+K\|B\|_{L^\infty(\Omega)}\int_\Gamma\big(|\nabla G_1|+|\nabla G_2|\big)\,dl. 
    $$
    Therefore, for each $\tau\inR,$
    \begin{align*}
        \int_{\{v=\tau\}\cap\Omega}|\nabla\tu-\nabla u|\,dl
        &\leq K\tK\int_{\{v=\tau\}\cap\Omega}\recip{\sigma_0}|J-\tJ|+|F-\tF|\,dl\nl
        &\hspace{0.6in}+K\|B\|_{L^\infty(\Omega)}\int_{\{v=\tau\}\cap\Omega}\big(|\nabla G_1|+|\nabla G_2|\big)\,dl.
    \end{align*}
    and therefore,
    \begin{align}\label{K1}
    \begin{split}
        \int_{\{v=\tau\}\cap\Omega}|\nabla\tv-\nabla v|\,dl
        &\leq\int_{\{v=\tau\}\cap\Omega}|\nabla\tu-\nabla u|+|F-\tF|\,dl\nl
        &\leq\int_{\{v=\tau\}\cap\Omega}\frac{K\tK}{\sigma_0}|J-\tJ|+(K\tK+1)|F-\tF|\,dl\nl
        &\hspace{0.6in}+K\|B\|_{L^\infty(\Omega)}\int_{\{v=\tau\}\cap\Omega}\big(|\nabla G_1|+|\nabla G_2|\big)\,dl.
    \end{split}
    \end{align}
    By \eqref{K1} and the coarea formula,
    \begin{align*}
        \frac{m}{\sigma_1}\|\nabla v-\nabla\tv\|_{L^1(\Omega)}
        &\leq \int_{\Omega}|\nabla v|\,|\nabla v-\nabla\tv|\,dx\nl
        &=\int_\R\int_{\{v=\tau\}\cap\Omega} |\nabla v-\nabla\tv|\,dl\,d\tau\nl
        &\leq\int_\R\int_{\{v=\tau\}\cap\Omega}\frac{K\tK}{\sigma_0}|J-\tJ|+(K\tK+1)|F-\tF|\,dl\,d\tau\nl
        &\hspace{0.6in}+K\|B\|_{L^\infty(\Omega)}\int_\R\int_{\{v=\tau\}\cap\Omega}\big(|\nabla G_1|+|\nabla G_2|\big)\,dl\,d\tau\nl
        &\leq\frac{M}{\sigma_0}\int_\R\int_{\{v=\tau\}\cap\Omega}\frac{K\tK}{\sigma_0}\frac{|J-\tJ|}{|\nabla v|}+(K\tK+1)\frac{|F-\tF|}{|\nabla v|}\,dl\,d\tau\nl
        &\hspace{0.8in}+\frac{KM\|B\|_{L^\infty(\Omega)}}{\sigma_0}\int_\R\int_{\{v=\tau\}\cap\Omega}\frac{|\nabla G_1|+|\nabla G_2|}{|\nabla v|}\,dl\,d\tau\nl
        &=\frac{M}{\sigma_0}\int_\Omega\frac{K\tK}{\sigma_0}|J-\tJ|+(K\tK+1)|F-\tF|\,dx\nl
        &\hspace{0.8in}+\frac{KM\|B\|_{L^\infty(\Omega)}}{\sigma_0}\int_\Omega \big(|\nabla G_1|+|\nabla G_2|\big)\,dx.
    \end{align*}
    Lastly,
    $$
        \int_\Omega \big(|\nabla G_1|+|\nabla G_2|\big)\,dx\leq 2C_1\big(\|J-\tJ\|_{L^1(\Omega)}^{1/2}+\|F-\tF\|_{L^1(\Omega)}^{1/2}\big)
    $$
    where $C_1=C_1\left(\Omega,M,k_2,\sigma_0,\sigma_2\right)$ is the constant obtained from Lemma \ref{lemma F3}. By Theorem \ref{thm F1} and the triangle inequality, we obtain the desired result. \hfill$\Box$

\begin{theorem}\label{thm F5}
    Let $n=3$ and suppose that $u$ and $\tu$ are admissible with $u|_{\partial\Omega}=\tu|_{\partial\Omega}=0.$ Suppose $F,\tF\in C^2(\bar\Omega),\,H\in C^1(\bar\Omega),\,\sigma,\tsigma\in C^2(\bar\Omega)$ and satisfy \eqref{FHineq3} and \eqref{sigmaineq3}. If the level sets of $v$ can be foliated to one-dimensional curves in the sense of Definition 3.4 in \cite{MO}, and the level sets of $v$ are well-structured in the sense of Definition 4.2 in \cite{MO} then
\begin{equation}
    \|\nabla u-\nabla\tu\|_{L^1(\Omega)}\leq C\|\,F-\tF\|_{L^1(\Omega)}^{1/4},
\end{equation}
for some constant $C=C(\Omega,m,M,k_2,K,\tK,\sigma_0,\sigma_1,\sigma_2)$ independent of $\tu$ and $\tsigma.$
\end{theorem}
{\bf Proof.}
    Similar to the proof of Theorem \ref{thm F4}, we can conclude that
    \begin{align}\label{K2}
    \begin{split}
        \int_{V_{\tau,r}}|\nabla v-\nabla\tv|\,dl
        &\leq\int_{V_{\tau,r}}\frac{K\tK}{\sigma_0}|J-\tJ|+(K\tK+1)|F-\tF|\,dl\nl
        &\hspace{0.6in}+K\|B\|_{L^\infty(\Omega)}\int_{V_{\tau,r}}\big(|G_1|+|G_2|+|G_3|\big)\,dl.
    \end{split}
    \end{align}
    where $V_{\tau,r}:=\{v=\tau\}\cap\{g_\tau=r\}\cap\Omega$ (see Definiton 3.4 in \cite{MO}) and $G=(G_1,G_2,G_3)$ is defined in \eqref{G2}. It follows from \eqref{K2} and the coarea formula that
    \allowdisplaybreaks
    \begin{align*}
        \frac{m}{\sigma_1}\|\nabla v-\nabla\tv\|_{L^1(\Omega)}
        &\leq\frac{MC_g}{\sigma_0c_g}\int_\R\int_\R\int_{V_{\tau,r}}\frac{K\tK}{\sigma_0}\cdot\frac{|J-\tJ|}{|\nabla g_\tau|\,|\nabla v|}+(K\tK+1)\frac{|F-\tF|}{|\nabla g_\tau|\,|\nabla v|}\,dl\,dr\,d\tau\nl
        &\hspace{0.44in}+\frac{K\|B\|_{L^\infty(\Omega)}}{c_g}\cdot C_g\cdot\frac{M}{\sigma_0}\int_\R\int_\R\int_{V_{\tau,r}}\frac{|\nabla G_1|+|\nabla G_2|+|\nabla G_3|}{|\nabla g_\tau|\,|\nabla v|}\,dl\,dr\,d\tau\nl
        &=\frac{MC_g}{\sigma_0c_g}\int_\Omega\frac{K\tK}{\sigma_0}|J-\tJ|+(K\tK+1)|F-\tF|\,dx\nl
        &\hspace{0.44in}+\frac{KMC_g\|B\|_{L^\infty(\Omega)}}{\sigma_0c_g}\int_\Omega\big(|\nabla G_1|+|\nabla G_2|+|\nabla G_3|\big)\,dx.
    \end{align*}
    Lastly,
    $$
        \int_\Omega \big(|\nabla G_1|+|\nabla G_2|+|\nabla G_3|\big)\,dx\leq 3C_1\Big(\|J-\tJ\|_{L^1(\Omega)}^{1/2}+(2k_2|\Omega|)^{1/4}\|F-\tF\|_{L^1(\Omega)}^{1/4}\Big)
    $$
    where $C_1=C_1\left(\Omega,M,k_2,\sigma_0,\sigma_2\right)$ is the constant obtained from Lemma \ref{lemma F3}.
    By Theorem \ref{thm F1} and the triangle inequality, we obtain the desired result.
    \hfill$\Box$

\begin{theorem}\label{thm F6}
Let $n=2$ and suppose that $u$ and $\tu$ are admissible with $u|_{\partial\Omega}=\tu|_{\partial\Omega}=0.$ Suppose $F,\tF\in C^2(\bar\Omega),\,H\in C^1(\bar\Omega),\,\sigma,\tsigma\in C^2(\bar\Omega)$ and satisfy \eqref{FHineq3} and \eqref{sigmaineq3}. If the level sets of $v$ are well-structured in the sense of Definition 4.2 in \cite{MO}, then
\begin{equation}
    \|\sigma-\tsigma\|_{L^1(\Omega)}\leq C\|\,F-\tF\|_{L^1(\Omega)}^{1/4},
\end{equation}
for some constant $C=C(\Omega,m,M,k_2,K,\tK,\sigma_0,\sigma_1,\sigma_2)$ independent of $\tsigma.$
\end{theorem}
{\bf Proof.}
    \begin{align*}
        \Int_\Omega|\sigma-\tsigma|\,dx
        =\Int_\Omega a\Frac{\big||\nabla v|-|\nabla\tv|\big|}{|\nabla v||\nabla\tv|}\,dx
        \leq\Frac{M\sigma_1^2}{m^2}\Int_\Omega|\nabla u-\nabla\tu|+|F-\tF|\,dx
    \end{align*}
    and we apply Theorem \ref{thm F4} as well as $\|F-\tF\|_{L^1(\Omega)}^{3/4}\leq(2k_2|\Omega|)^{3/4}$.
\hfill$\Box$ \\

Lastly, the following theorem follows from Theorem \ref{thm F5} and a calculation similar to the proof of Theorem \ref{thm F6}.

\begin{theorem}\label{thm F7}
Let $n=3$ and suppose that $u$ and $\tu$ are admissible with $u|_{\partial\Omega}=\tu|_{\partial\Omega}=0.$ Suppose $F,\tF\in C^2(\bar\Omega),\,H\in C^1(\bar\Omega),\,\sigma,\tsigma\in C^2(\bar\Omega)$ and satisfy \eqref{FHineq3} and \eqref{sigmaineq3}. If the level sets of $v$ can be foliated to one-dimensional curves in the sense of Definition 3.4 in \cite{MO} and the level sets of $v$ are well-structured in the sense of Definition 4.2 in \cite{MO}, then
\begin{equation}
    \|\sigma-\tsigma\|_{L^1(\Omega)}\leq C\|\,F-\tF\|_{L^1(\Omega)}^{1/4},
\end{equation}
for some constant $C=C(\Omega,m,M,k_2,K,\tK,\sigma_0,\sigma_1,\sigma_2)$ independent of $\tsigma.$
\end{theorem}

\section{Combined Stability with Respect to $a$, $F$, and $H$}
In this section, we combine the stability results established in the previous sections to obtain stability of minimizers of \eqref{mainmin0} under simultaneous perturbations of the parameters $a$, $F$, and $H$. More precisely, we derive quantitative estimates describing the dependence of minimizers on the combined variations of the weight function, the vector field, and the forcing term. These results provide a unified stability framework for the variational problem \eqref{mainmin0}.

\[
\begin{aligned}
I(w)&:=\Int_\Omega \left(a|Dw+F|+Hw\right),\\
I_1(w)&:=\Int_\Omega \left(a|Dw+F|+\tH w\right),\\
I_2(w)&:=\Int_\Omega \left(\ta|Dw+F|+\tH w\right),\\
\tilde I(w)&:=\Int_\Omega \left(\ta|Dw+\tF|+\tH w\right).
\end{aligned}
\]

Let $u$, $u_1$, $u_2$, and $\tu$ be minimizers of $I$, $I_1$, $I_2$, and $\tilde I$ over $BV_0(\Omega)$, respectively.

For $x\in\Omega,$ define
\[
\sigma(x):=\Frac{a(x)}{|Du(x)+F(x)|},\qquad
\sigma_1(x):=\Frac{a(x)}{|Du_1(x)+F(x)|},
\]
\[
\sigma_2(x):=\Frac{\ta(x)}{|Du_2(x)+F(x)|},\qquad
\tsigma(x):=\Frac{\ta(x)}{|D\tu(x)+\tF(x)|}.
\]

Note that
\[
|u-\tu|\leq |u-u_1|+|u_1-u_2|+|u_2-\tu|
\]
and
\[
|\sigma-\tsigma|
\leq
|\sigma-\sigma_1|
+
|\sigma_1-\sigma_2|
+
|\sigma_2-\tsigma|.
\]
Thus, we have introduced the intermediate minimizers so that the perturbations in $a$, $H$, and $F$ are applied one at a time. Consequently, the stability estimates established in the previous sections can be combined to obtain our main results.

\begin{theorem}
      Let $n=2.$ Then, under the assumptions of Theorem 2.11 in \cite{MO}, Theorem \ref{thm a2}, and Theorem \ref{thm F2}, the following estimate holds for some constant $C$ independent of $\tu$ and $\tsigma:$
    $$
        \|u-\tu\|_{L^1(\Omega)}\leq C\left(\|H-\tH\|_{L^\infty(\Omega)}^{\frac{1}{2}}+\|a-\ta\|_{L^\infty(\Omega)}^{\frac{1}{2}}+\max\left(\|f-\tf\|_{W^{1,1}(\Omega)}^{\frac{1}{2}},\,\|f-\tf\|_{W^{1,1}(\Omega)}\right)\right).
    $$
  \end{theorem}

  \begin{theorem}
      Let $n=3.$ Then, under the assumptions of Theorem 2.14 in \cite{MO}, Theorem \ref{thm a3}, and Theorem \ref{thm F3}, the following estimate holds for some constant $C$ independent of $\tu$ and $\tsigma:$
    $$
        \|u-\tu\|_{L^1(\Omega)}\leq C\left(\|H-\tH\|_{L^\infty(\Omega)}^{\frac{1}{2}}+\|a-\ta\|_{L^\infty(\Omega)}^{\frac{1}{2}}+\max\left(\|f-\tf\|_{W^{1,1}(\Omega)}^{\frac{1}{2}},\,\|f-\tf\|_{W^{1,1}(\Omega)}\right)\right).
    $$
  \end{theorem}

  \begin{theorem}
      Let $n=2.$ Then, under the assumptions of Theorem 2.18 in \cite{MO}, Theorem \ref{thm a4}, and Theorem \ref{thm F4}, the following estimate holds for some constant $C$ independent of $\tu$ and $\tsigma:$
    $$
        \|\nabla u-\nabla\tu\|_{L^1(\Omega)}\leq C\left(\|H-\tH\|_{L^\infty(\Omega)}^{\frac{1}{4}}+\|a-\ta\|_{L^\infty(\Omega)}^{\frac{1}{4}}+\|F-\tF\|_{L^1(\Omega)}^{\frac{1}{4}}\right).
    $$
  \end{theorem}

  \begin{theorem}
      Let $n=3.$ Then, under the assumptions of Theorem 2.19 in \cite{MO}, Theorem \ref{thm a5}, and Theorem \ref{thm F5}, the following estimate holds for some constant $C$ independent of $\tu$ and $\tsigma:$
    $$
        \|\nabla u-\nabla\tu\|_{L^1(\Omega)}\leq C\left(\|H-\tH\|_{L^\infty(\Omega)}^{\frac{1}{4}}+\|a-\ta\|_{L^\infty(\Omega)}^{\frac{1}{4}}+\|F-\tF\|_{L^1(\Omega)}^{\frac{1}{4}}\right).
    $$
  \end{theorem}

  \begin{theorem}
      Let $n=2.$ Then, under the assumptions of Theorem 2.20 in \cite{MO}, Theorem \ref{thm a6}, and Theorem \ref{thm F6}, the following estimate holds for some constant $C$ independent of $\tsigma:$
    $$
        \|\sigma-\tsigma\|_{L^1(\Omega)}\leq C\left(\|H-\tH\|_{L^\infty(\Omega)}^{\frac{1}{4}}+\|a-\ta\|_{L^\infty(\Omega)}^{\frac{1}{4}}+\|F-\tF\|_{L^1(\Omega)}^{\frac{1}{4}}\right).
    $$
  \end{theorem}

  \begin{theorem}
      Let $n=3.$ Then, under the assumptions of Theorem 2.21 in \cite{MO}, Theorem \ref{thm a7}, and Theorem \ref{thm F7}, the following estimate holds for some constant $C$ independent of $\tsigma:$
    $$
        \|\sigma-\tsigma\|_{L^1(\Omega)}\leq C\left(\|H-\tH\|_{L^\infty(\Omega)}^{\frac{1}{4}}+\|a-\ta\|_{L^\infty(\Omega)}^{\frac{1}{4}}+\|F-\tF\|_{L^1(\Omega)}^{\frac{1}{4}}\right).
    $$
  \end{theorem}

  \subsection{Numerical Simulation}

It was proven in \cite{CMO} that the following algorithm converges weakly to a minimizer, $u,$ of \eqref{mainmin0}.

\vspace{0.25in}
\textbf{Algorithm 1}

\vspace{0.3in}

Let $\lambda>0.$ Let $H\in L^\infty(\Omega)$ and initialize $b^0,d^0\in\left(L^2(\Omega)\right)^n.$ For $k\geq0,$
\begin{enumerate}
    \item Solve
    \begin{equation*}\label{Poisson}
        \Delta u^{k+1}=-\nabla\cdot\left(b^k-d^k\right)+\recip{\lambda}H,\enspace u^{k+1}\big|_{\partial\Omega}=0.
    \end{equation*}
    \item Compute
    $$
        d^{k+1}:=
        \begin{cases}
            \max\left\{\left|b^k+\nabla u^{k+1}+F\right|-\frac{a}{\lambda},0\right\}\frac{b^k+\nabla u^{k+1}+F}{\left|b^k+\nabla u^{k+1}+F\right|}-F &\text{if }\left|b^k+\nabla u^{k+1}+F\right|\neq0\\
            -F &\text{if }\left|b^k+\nabla u^{k+1}+F\right|=0
        \end{cases}.
    $$
    \item Let
    $$
        b^{k+1}:=b^k+\nabla u^{k+1}-d^{k+1}.
    $$
\end{enumerate}

In \cite{CMO}, the authors and Weitao Chen wrote MATLAB code to run numerical simulations demonstrating the convergence of Algorithm 1. We use a slightly modified version of the code to run the following simulation. We construct an example with a known minimizer and compare the approximate solution with the known exact solution. The following data gives the numerical errors of numerical simulations with Algorithm 1 (with $\lambda =1$) for mesh size $h=1/100.$ The iterations are stopped when 
\[|u^{k+1}-u^k|/|u^{k+1}|<1\times10^{-7}.\]

We examine the effect of noise in our simulation. The noise model we used is a simple stochastic model $\tH = H + \gamma * R,$ where $R$ is normally distributed pseudo-random matrix of the order as $H$ with mean zero and standard deviation of one, and $\gamma > 0$ is the model standard deviation chosen as $\gamma = \delta * ||H||/||R||,$ where $\delta$
is the noise level. Similarly, we add noise to $a$ and $F$ in order to obtain $\ta$ and $\tF.$ Replacing $H$ with $\tH,$ $a$ with $\ta,$ and $F$ with $\tF$ in Algorithm 1, we obtain the following approximations of $\tu$ and compare them with $u.$

\begin{example}
    Let $\Omega=(0,1)\times(0,1)\subset\R^2$ and $u:=xy(1-x)(1-y)$ so that $u|_{\partial\Omega}=0.$ Let $F:=-\nabla u+(1,x+y)$ so that $\nabla u+F=(1,x+y)$ and let $a:=|\nabla u+F|=\sqrt{1+(x+y)^2}.$ Let $H:=\nabla\cdot(1,x+y)=1$ so that \eqref{EL} holds. 
\end{example}

\begin{table}[ht]
\centering
    
    \label{tab:exampleH1}
\begin{tabular}{ | m{3.8cm}| m{5cm} | m{3.85cm} |} 
  \hline
  \vspace{0.1cm}
  Low Noise $(\delta =0.01)$ & Moderate Noise $(\delta=0.035)$ & High Noise $(\delta =0.06)$ \\ 
  \hline
  \vspace{0.1cm}
  $0.0260$ & $0.0978$ & $0.1718$\\ 
  \hline
  \end{tabular}
  \caption{Relative $L^2$ errors for Algorithm 1 with $h=1/100$ and $Tol=1\times10^{-7}$ and increasing noise levels.}\vspace{0.1in}
  \end{table}

\begin{center}
    \includegraphics[width=4.2in]{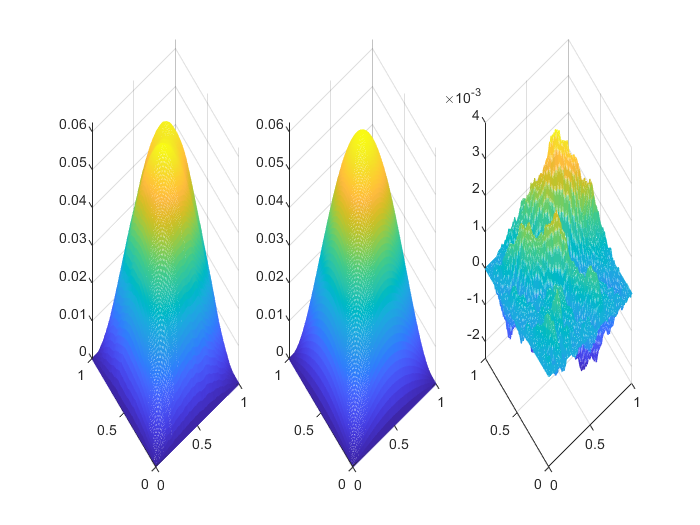}
        \begin{minipage}{15cm}

        \small  Figure 1: Numerical approximation $\tilde{u}^{405}$ (Left), Exact $u$ (Middle), and Error\\ $\tilde{u}^{405}-u$ (Right) with low noise. Maximum error: $0.0036.$

    \end{minipage}
\end{center}

\begin{center}
    \includegraphics[width=4.2in]{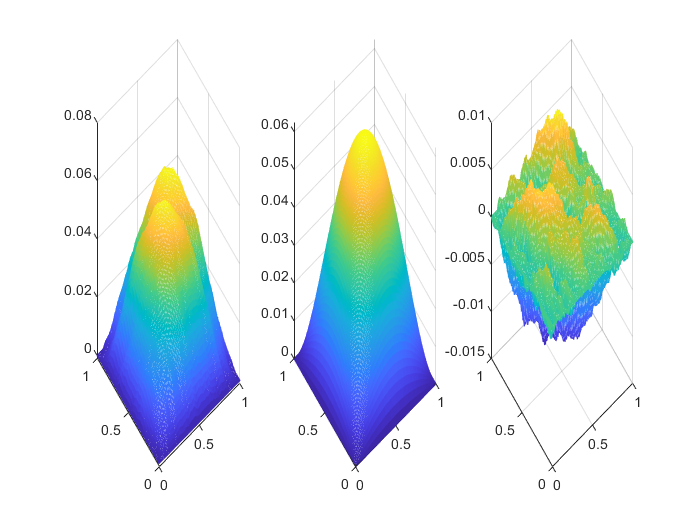}
        \begin{minipage}{15cm}

        \small  Figure 2: Numerical approximation $\tilde{u}^{446}$ (Left), Exact $u$ (Middle), and Error\\ $\tilde{u}^{446}-u$ (Right) with moderate noise. Maximum error: $0.0121.$

    \end{minipage}
\end{center}

\begin{center}
    \includegraphics[width=4.2in]{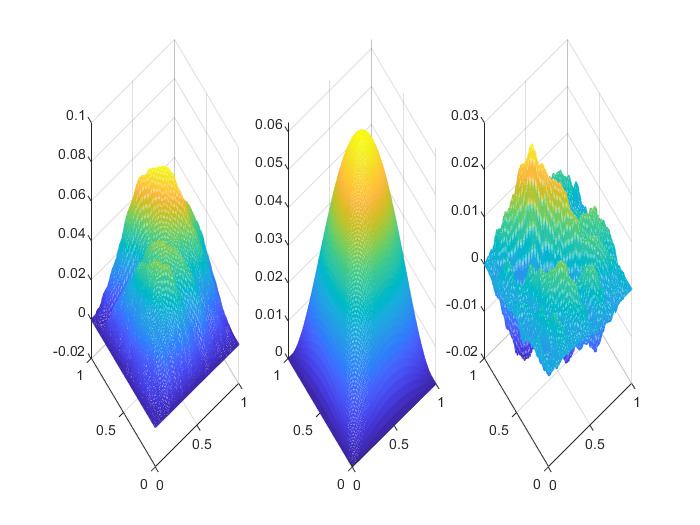}
        \begin{minipage}{15cm}

        \small  Figure 3: Numerical approximation $\tilde{u}^{449}$ (Left), Exact $u$ (Middle), and Error\\ $\tilde{u}^{449}-u$ (Right) with high noise. Maximum error: $0.0262.$

    \end{minipage}
\end{center}

\begin{center}
    \includegraphics[width=4.2in]{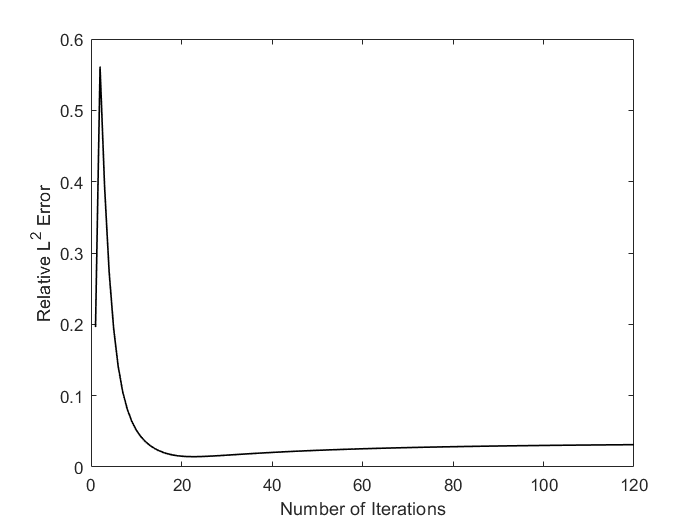}
        \begin{minipage}{15cm}

        \small  Figure 4: Rate of convergence for Algorithm 1 with low noise.

    \end{minipage}
\end{center}

\begin{center}
    \includegraphics[width=4.2in]{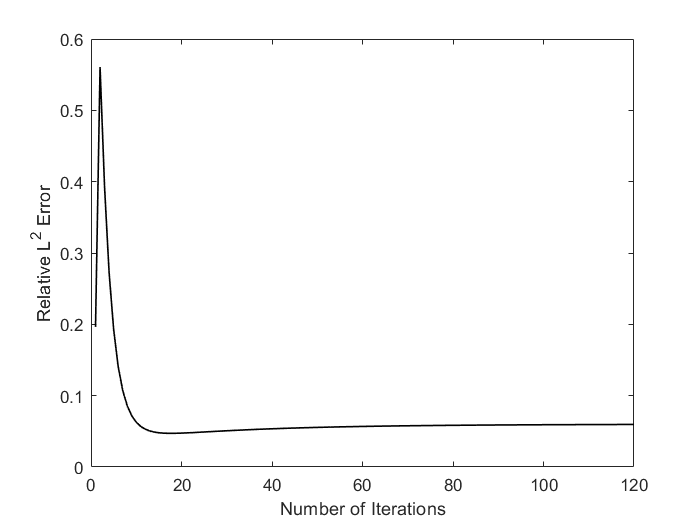}
        \begin{minipage}{15cm}

        \small  Figure 5: Rate of convergence for Algorithm 1 with moderate noise.

    \end{minipage}
\end{center}

\begin{center}
    \includegraphics[width=4.2in]{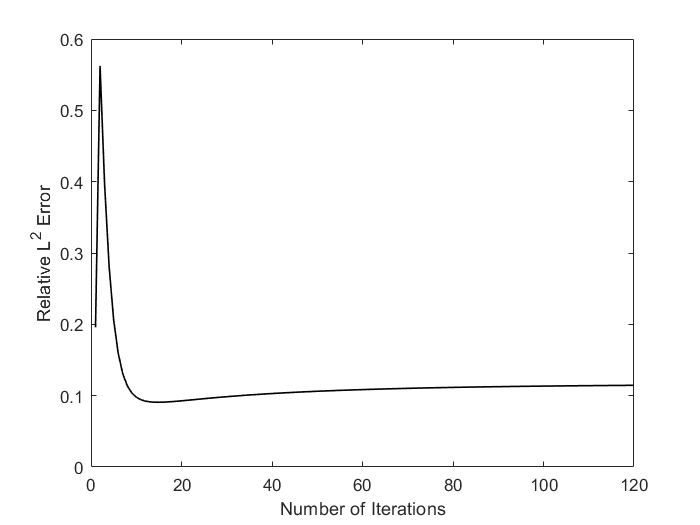}
        \begin{minipage}{15cm}

        \small  Figure 6: Rate of convergence for Algorithm 1 with high noise.

    \end{minipage}
\end{center}

\vspace{1cm}
\noindent
\textbf{Conflict of Interest.} 
The authors declare that they have no conflicts of interest relevant to this article. \\

\noindent
\textbf{Data availability.} This study is theoretical and includes numerical simulations. No external datasets were used. The data generated during the simulations are available from the authors upon reasonable request.

\end{document}